\documentclass[12pt,final]{article}
\usepackage[margin=2.5cm,top=1.5cm]{geometry}
\usepackage[all,arc]{xy}
\usepackage{amsmath,amsthm,amssymb,enumerate}
\usepackage{color}
\newtheorem{con}{Conjecture}[section]
\newtheorem{theo}[con]{Theorem}

\newtheorem{cor}[con]{Corollary}
\newtheorem{lem}[con]{Lemma}
\newtheorem{prop}[con]{Proposition}
\theoremstyle{definition}

\theoremstyle{remark}

\newcommand{\Romannum}[1]{\uppercase\expandafter{\romannumeral #1}}

\numberwithin{equation}{section}
 \DeclareMathOperator{\sgn}{sgn}

\newcommand\keywordsname{Key words}
\newcommand\AMSname{AMS subject classifications}

\newenvironment{@abssec}[1]{%
     \if@twocolumn
       \section*{#1}%
     \else
       \vspace{.05in}\footnotesize
       \parindent .2in
         {\upshape\bfseries #1. }\ignorespaces
     \fi}
     {\if@twocolumn\else\par\vspace{.1in}\fi}

\begin{document}
\title{\vspace*{3cm} Further results on the nullity of signed graphs\footnote{Research supported by National Natural Science Foundation of China
(No.10901061), the Zhujiang Technology New Star Foundation of
Guangzhou (No.2011J2200090), and Program on International Cooperation and Innovation, Department of Education,
Guangdong Province (No.2012gjhz0007).}}
\author{Yu Liu\footnote{{\it{Corresponding author:\;}}ylhua@scnu.edu.cn} \qquad
Lihua You\footnote{yumeiren2012@126.com} }
\vskip.2cm
\date{{\small
$^{}$School of Mathematical Sciences, South China Normal University\\ Guangzhou, 510631, P.R. China\\
}} \maketitle

\vspace{-7mm}

\begin{abstract}
The nullity of a graph is the multiplicity of the eigenvalues zero in its spectrum.
A signed graph is a graph with a sign attached to each of its edges.
In this paper, we obtain the coefficient theorem of the characteristic polynomial of a signed graph,
give two formulae on the nullity of signed graphs with cut-points.
 As applications of the above results, we investigate the nullity of the bicyclic signed graph $\Gamma(\infty(p,q,l))$,
 obtain  the nullity set of unbalanced bicyclic signed graphs, and thus determine the nullity set of bicyclic signed graphs.
\end{abstract}

\baselineskip=0.30in

\section{Introduction}
\hskip.6cm Let $G=(V, E)$ be a simple graph with vertex set $V=V(G)=\{v_1,v_2, \ldots, v_n\}$ and edge set $E=E(G)$.
The adjacency matrix $A=A(G)=(a_{ij})_{n\times n}$ of $G$ is defined as follows: $a_{ij}=1$ if there exists an edge joining $v_i$ and $v_j$, and $a_{ij}=0$ otherwise. The nullity of a simple graph $G$ is  the multiplicity of the eigenvalue zero in the spectrum of $A(G)$,  denoted by $\eta(G)$.
The rank of $G$ is referred to the rank of $A(G)$, and denoted by $r(G)$.
Clearly, $\eta(G)+r(G)=n$ if $G$ has $n$ vertices.

A signed graph $\Gamma(G)=(G,\sigma)$  is a graph with a sign attached to each of its edges,
consists of a simple graph $G=(V,E)$, referred to as its underlying graph, and a mapping $\sigma:E\rightarrow\{+,-\}$,  the edge labeling.
To avoid confusion, we also write $V(\Gamma(G))$ instead of $V$, $E(\Gamma(G))$ instead of $E$, and $E(\Gamma(G))={E}^{\sigma}$.

The adjacency matrix of $\Gamma(G)$ is $A(\Gamma(G))=(a_{ij}^{\sigma})$ with $a_{ij}^{\sigma}=\sigma(v_{i}v_{j})a_{ij}$, where $(a_{ij})$ is the adjacency matrix of the underlying graph $G$. In the case of $\sigma=+$, which is an all-positive edge labeling, $A(G,+)$ is exactly the classical adjacency matrix of $G$. So a simple graph is always assumed as a signed graphs with all edges positive.

The nullity of a signed graph $\Gamma(G)$ is defined as the multiplicity of the eigenvalue zero in the spectrum of $A(\Gamma(G))$, and is denoted by $\eta(\Gamma(G))$.
The rank of $\Gamma(G)$ is referred to the rank of $A(\Gamma(G))$, and denoted by $r(\Gamma(G))$.
Surely, $\eta(\Gamma(G))+r(\Gamma(G))=n$ if $\Gamma(G)$ has $n$ vertices.

Let $\Gamma(G)=(G,\sigma)$ be a signed graph, $\Gamma(C)$ be a signed cycle  of $\Gamma(G)$.
The sign of $\Gamma(C)$ is defined by $\sgn(\Gamma(C))=\prod\limits_{e\in \Gamma(C)}\sigma(e)$.
The cycle $\Gamma(C)$  is said to be positive or negative if $\sgn(\Gamma(C))=+$ or $\sgn(\Gamma(C))=-$.
 A signed graph is said to be balanced if all its cycles are positive, or equivalently,
 all cycles have  even number of negative edges; otherwise it is called unbalanced.

About the nullity of simple graphs and its applications, there are many known results, see  \cite{2007}, \cite{2011}, \cite{2012}, \cite{1957}, \cite{20072}, \cite{2009}, \cite{2010}, \cite{20122}, \cite{2001}, \cite{20112}, \cite{20092}, \cite{2008a}, \cite{2008c}, \cite{2008b}, \cite{1950} for details. About signed graphs, it was introduced by Harary \cite{1953} in connection with the study of the theory of social balance in social psychology. More results of signed graphs and their applications, see \cite{1994}, \cite{1980}, \cite{19802}, \cite{20131}, \cite{20132}, \cite{1995}, \cite{1953}, \cite{2003}, \cite{2005}, \cite{1987}, \cite{19942}, \cite{1999}, \cite{2008d}.

In this paper, we obtain the coefficient theorem of the characteristic polynomial of a signed graph,
give two formulae on the nullity of signed graphs with cut-points.
 As applications of the above results, we investigate the nullity of the bicyclic signed graph $\Gamma(\infty(p,q,l))$,
 obtain  the nullity set of unbalanced bicyclic signed graphs, and thus determine the nullity set of bicyclic signed graphs.

\section{The coefficients of $P_{\Gamma(G)}(\lambda)$}
\hskip.6cm In this section, we obtain the coefficients theorem of the characteristic polynomial of a signed graph $\Gamma(G)$, $P_{\Gamma(G)}(\lambda)$,
and the nullity of a signed cycle $\Gamma(C_{n})$ by using the coefficients theorem.

\begin{theo}\label{thm21}{\rm (\cite{1980})}
Let $$P_{G}(\lambda)=\mid\lambda I_n-A\mid=\lambda^{n}+a_{1}\lambda^{n-1}+\cdots+a_{n}$$ be the characteristic polynomial of an arbitrary undirected weighted multigraph $G$.

Call an ``elementary figure":  %\rq\rq
a) the graph $K_{2}$,
 or
b) every graph $C_{q}(q\geq1)$ (loops being included with $q=1$).

Call a \lq\lq basic figure" $U$ every graph all of whose component are elementary figures; let $p(U), c(U)$ be the number of components and the number of cycles contained in $U$, respectively, and let $\mathcal{U}_{i}$ denoted the set of all basic figures contained in $G$ having exactly $i$ vertices.
Then for any $i\in\{1,2,\ldots, n\}$,
$$a_{i}=\sum\limits_{U\in\mathcal{U}_{i}} (-1)^{p(U)}\cdot2^{c(U)}\cdot\prod(U), \quad \prod(U)=\prod\limits_{e\in E(U)} (w(e))^{\xi(e; U)},$$
\noindent where $E(U)$ is the set of edges of $U$, $w(e)$ is the weight of the edge $e$, and
$$ \xi(e;U)=\left\{
\begin{array}{lll}
1,     &     &  {\mbox{if } e \mbox{ is contained in some cycle of } U;}\\
2,     &     &  {\mbox{otherwise.} }
\end{array}
\right. $$
\end{theo}
\begin{cor}\label{cor21}
Let $\Gamma(G)$ be a signed graph on $n$ vertices and
$$P_{\Gamma(G)}(\lambda)=\mid\lambda I_n-A(\Gamma(G))\mid=\lambda^{n}+a_{1}\lambda^{n-1}+\cdots+a_{n}$$ be the characteristic polynomial of $A(\Gamma(G))$.
Then for any $i\in\{1,2,\ldots, n\}$,
$$a_{i}=\sum_{\Gamma(U)\in\Gamma(\mathcal{U}_{i})} (-1)^{p(\Gamma(U))+s(\Gamma(U))}\cdot2^{c(\Gamma(U))},$$
\noindent where $s(\Gamma(U))$ is the number of negitive edges in cycle of $\Gamma(U)$, other notations are similar to  Theorem \ref{thm21}.
\end{cor}
\begin{proof}
Since $\Gamma(G)$ is a signed graph, $w(e)=+1$ or $-1$, then $\prod(\Gamma(U))=\prod\limits_{e\in E(\Gamma(U))} (-1)^{s(\Gamma(U))}$.
Thus the result follows from Theorem \ref{thm21}.
\end{proof}

\begin{lem}\label{lem21}
{\rm (1)  (\cite{1980})} Let $\Gamma(C_{n})$ be a balanced  cycle. Then $\eta(\Gamma(C_{n}))=2$ if $n\equiv \ 0(\ \bmod \ 4)$ and $\eta(\Gamma(C_{n}))=0$ otherwise.
%\end{lem}
%\begin{lem}\label{lem22}

{\rm (2)  (\cite{20131})} Let $\Gamma(C_{n})$ be an unbalanced signed cycle. Then $\eta(\Gamma(C_{n}))=2$ if $n\equiv \ 2(\ \bmod \ 4)$ and $\eta(\Gamma(C_{n}))=0$ otherwise.
\end{lem}

Let $s=s(\Gamma(C_{n}))$ be the number of negative edges of $\Gamma(C_{n})$. It is clear that $\Gamma(C_{n})$ is balanced if and only if $s\equiv 0(\ \bmod \ 2)$; $\Gamma(C_{n})$ is unbalanced if and only if  $s\equiv 1(\ \bmod \ 2)$. Then Lemma \ref{lem21} is equivalent to the following Theorem \ref{thm22}. We will give a new proof  by Corollary \ref{cor21}.
\begin{theo}\label{thm22}
Let $\Gamma(C_{n})=(C_n, \sigma)$ be a signed cycle on $n$ vertices, $s$ is the number of negative edges of $\Gamma(C_{n})$. Then
$$ \eta({\Gamma(C_{n})})=\left\{
\begin{array}{ll}
2,         &  {\mbox{if } n \equiv \ 0(\ \bmod \ 4), s\equiv \ 0(\ \bmod \ 2);}\\
2,         &  {\mbox{if } n \equiv \ 2(\ \bmod \ 4), s\equiv \ 1(\ \bmod \ 2);}\\
0,         &  {\mbox{if } n \equiv \ 1(\ \bmod \ 2);}\\
0,         &  {\mbox{if } n \equiv \ 0(\ \bmod \ 4), s\equiv \ 1(\ \bmod \ 2);}\\
0,         &  {\mbox{if } n \equiv \ 2(\ \bmod \ 4), s\equiv \ 0(\ \bmod \ 2).}
\end{array}
\right. $$
\end{theo}
\begin{proof}
By Corollary \ref{cor21},  for any $i\in\{1,2,\ldots, n\}$, we have $$a_{i}=\sum_{\Gamma(U)\in\Gamma(\mathcal{U}_{i})} (-1)^{p(\Gamma(U))+s(\Gamma(U))}\cdot2^{c(\Gamma(U))}.$$

{\bf Case 1: } $n \equiv \ 1(\ \bmod \ 2)$.

 Clearly, $\Gamma(U)=\Gamma(C_{n})$. Thus $a_{n}=2\cdot (-1)^{s+1}\not= 0$.

{\bf Case 2: } $n \equiv \ 0(\ \bmod \ 2)$.

Clearly, $\Gamma(U)=\Gamma(C_{n})$ or $\Gamma(U)=\frac{n}{2}\Gamma(K_{2})$, and there exist two basic figures $\frac{n}{2}\Gamma(K_{2})$ in $\Gamma(C_{n})$. Then
$$ a_{n}=2\cdot (-1)^{s+1}+2\cdot (-1)^{\frac{n}{2}}=\left\{
\begin{array}{lll}
0,      &     &  {\mbox{if } n \equiv \ 0(\ \bmod \ 4), s\equiv \ 0(\ \bmod \ 2);}\\
4,      &     &  {\mbox{if } n \equiv \ 0(\ \bmod \ 4), s\equiv \ 1(\ \bmod \ 2);}\\
-4,     &     &  {\mbox{if } n \equiv \ 2(\ \bmod \ 4), s\equiv \ 0(\ \bmod \ 2);}\\
0,      &     &  {\mbox{if } n \equiv \ 2(\ \bmod \ 4), s\equiv \ 1(\ \bmod \ 2).}\\
\end{array}
\right. $$

If $a_{n}\neq0$, then $\eta(\Gamma(C_{n}))=0$.

If $a_{n}=0$, then we consider $a_{n-1}$ and $a_{n-2}$. Since $n$ is even, it's clear that $a_{n-1}=0$ and $a_{n-2}\neq0$ by Corollary \ref{cor21}. Thus $\eta(\Gamma(C_{n}))=2$.
\end{proof}

Similarly,  by Corollary \ref{cor21}, we have
\begin{prop}\label{prop21}
Let $\Gamma(P_{n})=(P_n,\sigma)$ be a signed path on $n$ vertices. Then
$$ \eta(\Gamma(P_{n}))=\left\{
\begin{array}{lll}
1,     & {\mbox{if } n \mbox{ is odd}; }\\
0,     & {\mbox{if } n \mbox{ is even}. }
\end{array} \right. $$
\end{prop}

\section{The nullity of a signed graph with cut-points}
\hskip.5cm In this section, we deduce two concise formulae on the nullity of  signed graphs with cut-points by similar technique which applied in \cite{20122}.

We first introduce some concepts and notations.

Let $G$ be a simple graph with vertex set $V=V(G)$.
For a nonempty subset $U$ of $V$, the subgraph with vertex set $U$ and edge set consisting of those pairs of vertices that
are edges in $G$ is called the induced subgraph of $G$, denoted by $G[U]$.
Denote by $G-U$, where $U\subseteq V$,
the graph obtained from $G$ by removing the vertices of $U$ together with all edges incident to them.
Sometimes we use the notation $G-G_1$ instead of $G-V(G_1)$ if $G_1$ is an induced subgraph of $G$.
For an induced subgraph $G_1$ (of $G$) and $v\in G-G_1$, the induced subgraph $G[V(G_1)\cup \{v\}]$ is simply written  as $G_1+v$.
The vertex $v\in V$ is called a cut-point of $G$ if the resultant graph $G-v$ is disconnected.

Let $A$ be the adjacency matrix of a graph $G=(V(G),E(G))$ on $n$ vertices.
 For $U\subseteq V(G)$, $W\subseteq V(G)$, denote by $A[U,W]$ the submatrix of
$A$ with rows corresponding to the vertices of $U$ and columns corresponding to the vertices of $W$.
For simplify, the submatrix $A[U, U]$ is written as $A[U]$.
 For convenience, we usually write $A[G_1, G_2]$  instead of the standard $A[V(G_1), V(G_2)]$  for the two induced subgraphs $G_1$ and $G_2$ of $G$.
In particular, denote by $A[v, G]$ the row vector of $A$ corresponding to the vertex $v$ and by $A[v, G_i]$
the subvector of $A[v, G]$ corresponding to the vertices of $G_i$. We refer to Cvetkovi\'{c}  et al. \cite{1980} for more
terminologies and notations not defined here.

The following lemma is obvious.

\begin{lem}\label{lem31}
Let  $G=G_1\cup G_2\cup \cdots \cup G_s$, and  $\Gamma(G)$ be a signed graph,   where $G_1, G_2, \ldots, G_s$ are the connected components of $G$. Then $\Gamma(G_{1}), \Gamma(G_{2}),\ldots, \Gamma(G_{s})$ are the connected components of $\Gamma(G)$,
$r(\Gamma(G))=\sum\limits_{i=1}^{s}r(\Gamma(G_{i})),$ and $\eta(\Gamma(G))=\sum\limits_{i=1}^{s}\eta(\Gamma(G_{i})).$
\end{lem}

\begin{theo}\label{thm31}
Let $\Gamma(G)$ be a connected signed graph on $n$ vertices, $v$ be a cut-point of $\Gamma(G)$, and $\Gamma(G_{1}), \Gamma(G_{2}),\ldots, \Gamma(G_{s})$ are all the components of $\Gamma(G-v)$. If there exists a component, say $\Gamma(G_{1})$, among $\Gamma(G_{1}),\Gamma(G_{2}),\ldots, \Gamma(G_{s})$ such that $\eta(\Gamma(G_{1}))=\eta(\Gamma(G_{1}+v))+1 $. Then $\eta(\Gamma(G))=\eta(\Gamma(G-v))-1=\sum\limits_{i=1}^{s}{\eta(\Gamma(G_{i}))}-1$.
\end{theo}
\begin{proof}
Let $A=A(\Gamma(G))$ be the adjacency matrix of $\Gamma(G)$. For each $i$, denote by $A[\Gamma(G_{i})]$ the adjacency matrix of the subgraph $\Gamma(G_{i})$ and by $A[v,\Gamma(G_{i})]$ the subvector of $A[v,\Gamma(G)]$ corresponding to the vertices of $\Gamma(G_{i})$, then the matrix $A$ can be partitioned as:
\begin{gather*}
\begin{bmatrix} 0 & A[v,\Gamma(G_{1})] & A[v,\Gamma(G_{2})] & \cdots & A[v,\Gamma(G_{s})] \\ A[\Gamma(G_{1}),v] & A[\Gamma(G_{1})] & 0 & \cdots & 0 \\ A[\Gamma(G_{2}),v] & 0 & A[\Gamma(G_{2})] & \cdots & 0 \\ \cdots & \cdots & \cdots & \cdots & \cdots \\ A[\Gamma(G_{s}),v] & 0 & 0 & \cdots & A[\Gamma(G_{s})]
\end{bmatrix},
\end{gather*}
where $A[\Gamma(G_{i}),v]=A[v,\Gamma(G_{i})]^{T}(i=1,2,\ldots,s)$, the transpose of $A[v,\Gamma(G_{i})]$, for each $i$.

Note that $\eta(\Gamma(G_{1}))=\eta(\Gamma(G_{1}+v))+1$, then $r(\Gamma(G_{1}+v))=r(\Gamma(G_{1}))+2$, thus the row vector $A[v,\Gamma(G_{1})]$ is not linear combination of the row vector of $A[\Gamma(G_{1})]$, therefore the row vector $A[v,\Gamma(G)]$ is not linear combination of any other row vectors of $A[\Gamma(G)]$. Since $A$ is a symmetric matrix, the column vector $A[\Gamma(G),v]$ is not linear combination of any other column vectors of $A$, which implies that $r(\Gamma(G))=r(\Gamma(G-v))+2$. Then by Lemma \ref{lem31},

$\eta(\Gamma(G))=n-r(\Gamma(G))=n-r(\Gamma(G-v))-2=\eta(\Gamma(G-v))-1=\sum\limits_{i=1}^{s}{\eta(\Gamma(G_{i}))}-1$ .
\end{proof}

\begin{theo}\label{thm32}
Let $\Gamma(G)$ be a connected signed graph on $n$ vertices, let $v$ be a cut-point of $\Gamma(G)$ and $\Gamma(G_{1})$ be a component of $\Gamma(G-v)$. If $\eta(\Gamma(G_{1}))=\eta(\Gamma(G_{1}+v))-1$, then $\eta(\Gamma(G))=\eta(\Gamma(G_{1}))+\eta(\Gamma(G-G_{1}))$.
\end{theo}
\begin{proof}
Let $A=A(\Gamma(G))$ be the adjacency matrix of $\Gamma(G)$. Then
\begin{gather*}
A=\begin{bmatrix} A[\Gamma(G_{1})] & A[\Gamma(G_{1}),v] & 0 \\ A[v,\Gamma(G_{1})] & 0 & A[v,\Gamma(G-G_{1}-v)] \\ 0 & A[\Gamma(G-G_{1}-v),v] & A[\Gamma(G-G_{1}-v)]
\end{bmatrix}.
\end{gather*}

Because $\eta(\Gamma(G_{1}))=\eta(\Gamma(G_{1}+v))-1$, $r(\Gamma(G_{1}+v))=r(\Gamma(G_{1}))$ and thus the row vector
$A[v,\Gamma(G_{1}+v)]=[A[v,\Gamma(G_{1})] \quad 0]$ is linear combination of the row vectors of $A[\Gamma(G_{1}),\Gamma(G_{1}+v)]$. Similarly, the column vector $A[\Gamma(G_{1}+v),v]$ is linear combination of the column vector of $A[\Gamma(G_{1}+v), \Gamma(G_{1})]$. Thus $B$ can be obtained from $A$ by row and column linear transformations, where
\begin{gather*}
B=\begin{bmatrix} A[\Gamma(G_{1})] & 0 & 0 \\ 0 & 0 & A[v,\Gamma(G-G_{1}-v)] \\ 0 & A[\Gamma(G-G_{1}-v),v] & A[\Gamma(G-G_{1}-v)]
\end{bmatrix}.
\end{gather*}
It's easy to see $B$ is the adjacency matrix of the union of $\Gamma(G_{1})$ and $\Gamma(G-G_{1})$. Then we have $r(A)=r(B)$, which implies that $\eta(\Gamma(G))=n-r(A)=n-r(B)=\eta(B)=\eta(\Gamma(G_{1}))+\eta(\Gamma(G-G_{1}))$.
\end{proof}

\section{The nullity of the bicyclic signed graph $\Gamma(\infty(p,q,l))$}

\hskip.6cm  %In this section, we obtain the nullity of the bicyclic signed graph $\Gamma(\infty(p,q,l))$.
Firstly, we introduce  some definitions and notation which will used in the following.

A bicyclic graph is a simple connected graph in which the number of edges equals the number of vertices plus one.
There are two basic bicyclic graphs: $\infty$-graph and $\Theta$-graph. An $\infty$-graph, denoted by
$\infty(p, q, l)$, is obtained from two vertex-disjoint cycles $C_p$ and $C_q$ by connecting one vertex
of $C_p$ and one of $C_q$ with a path $P_l$ of length $l-1$
(in the case of $l=1$, identifying the above two vertices);
and a $\Theta$-graph, denoted by $\Theta(p, q, l)$, is a union of three internally disjoint
paths $P_{p+1}, P_{q+1}, P_{l+1}$  of length $p, q, l$ respectively, with common end vertices, where
$p, q, l\geq 1$ and at most one of them is 1. Observe that any bicyclic graph $G$ is obtained
from an $\infty$-graph or a $\Theta$-graph (possibly) by attaching trees to some of its vertices.

%A signed $\infty$-graph, denoted by $\Gamma(\infty(k,p,l))$, is obtained from two vertex-disjoint signed cycles $\Gamma(C_{k})$ and $\Gamma(C_{l})$ by connecting one %vertex of $\Gamma(C_{k})$ and one of $\Gamma(C_{l})$ with a signed path $\Gamma(P_{p})$ of length $p-1$ (in the case $p=1$, identifying the above two vertices); and a %signed $\Theta$-graph, denoted by $\Gamma(\Theta(k,p,l))$, is a union of three internally disjoint signed paths $\Gamma(P_{k+1}),\Gamma(P_{p+1}),\Gamma(P_{l+1})$ of %length $k,p,l$. respectively with common end vertices. where $k,p,l\geq1$ and at most one of them is $1$.

\begin{lem}\label{lem11}{\rm (\cite{20131})}
Let $\Gamma(G)$ be a signed graph containing a pendant vertex, and let $\Gamma(H)$ be the induced subgraph of $\Gamma(G)$ obtained by deleting this pendant vertex together with the vertex adjacent to it. Then $\eta(\Gamma(G))=\eta(\Gamma(H))$.
\end{lem}

\begin{theo}\label{thm41}
Let $p,q,l$ be integers with $p, q\geq 3, l\geq 1$ and $G=\infty(p,q,l)$.

{\rm (1) }If $p$ and $q$ are odd, then
$$ \eta(\Gamma(G))=\left\{
\begin{array}{lll}
0,        & {\mbox{if } l \mbox{ is even}; }\\
0,        & {\mbox{if }  l \mbox{ is odd  and } s(\Gamma(C_{p}))-s(\Gamma(C_{q}))+\frac{q-p}{2}\equiv \ 0(\ \bmod \ 2); }\\
1,         & {\mbox{if } l=1  \mbox{ and } s(\Gamma(C_{p}))-s(\Gamma(C_{q}))+\frac{q-p}{2}\equiv \ 1(\ \bmod \ 2); }\\
\geq1,    & {\mbox{if } l(\geq 3) \mbox{ is odd and } s(\Gamma(C_{p}))-s(\Gamma(C_{q}))+\frac{q-p}{2}\equiv \ 1(\ \bmod \ 2). }\\
\end{array} \right. $$

{\rm (2) } If $p$ and $q$ have different parities, without loss of generality, let $p$ be even, then
$$ \eta(\Gamma(G))=\left\{
\begin{array}{lll}
0,      & {\mbox{if }\eta(\Gamma(C_{p}))=0; }\\
1,      & {\mbox{if }\eta(\Gamma(C_{p}))=2. }\\
\end{array} \right. $$

{\rm (3) } If $p$ and $q$ are even, then
$$ \eta(\Gamma(G))=\left\{
\begin{array}{lll}
3,      & {\mbox{if } l\mbox{ is odd}, \eta(\Gamma(C_{p}))=\eta(\Gamma(C_{q}))=2; }\\
1,       & {\mbox{if } l\mbox{ is odd}, \eta(\Gamma(C_{p}))\cdot \eta(\Gamma(C_{q}))=0; }\\
2,          & {\mbox{if } l\mbox{ is even}, \eta(\Gamma(C_{p}))=2 \mbox{ or }\eta(\Gamma(C_{q}))=2; }\\
0,         & {\mbox{if } l\mbox{ is even}, \eta(\Gamma(C_{p}))=\eta(\Gamma(C_{q}))=0. }\\
\end{array} \right. $$
\end{theo}

\begin{proof}
{\bf Case 1: } Both $p$ and $q$ are odd.

By Corollary \ref{cor21}, we know
$$a_{i}=\sum_{\Gamma(U)\in\Gamma(\mathcal{U}_{i})} (-1)^{p(\Gamma(U))+s(\Gamma(U))}\cdot2^{c(\Gamma(U))}\quad (i=1,2,\cdots £¬n).$$

{\bf Subcase 1.1: } $l$ is even.
\begin{eqnarray*}
a_{n}& = & (-1)^{s(\Gamma(C_{p}))+s(\Gamma(C_{q}))+\frac{l+2}{2}}\times 2^{2} +(-1)^{\frac{p-1}{2}+\frac{q-1}{2}+\frac{l}{2}}\times 2^{0}\neq 0.
\end{eqnarray*}

\noindent So $\eta(\Gamma(G))=0$.

{\bf Subcase 1.2: } $l$ is odd.
\begin{eqnarray*}
a_{n}= (-1)^{s(\Gamma(C_{p}))+\frac{l+q}{2}}\times 2^{1} +(-1)^{s(\Gamma(C_{q}))+\frac{l+p}{2}}\times 2^{1}
%& = & 2[(-1)^{s(\Gamma(C_{p}))+\frac{l+q}{2}}+(-1)^{s(\Gamma(C_{q}))+\frac{l+p}{2}}]
\end{eqnarray*}

\hskip3.5cm $ \left\{
\begin{array}{lll}
=0,     & {\mbox{if } s (\Gamma(C_{p}))-s(\Gamma(C_{q}))+\frac{q-p}{2}\equiv \ 1(\ \bmod \ 2); }\\
\neq0,   & {\mbox{if } s (\Gamma(C_{p}))-s(\Gamma(C_{q}))+\frac{q-p}{2}\equiv \ 0(\ \bmod \ 2). }\\
\end{array} \right. $

\vskip.2cm

 Clearly, if $a_{n}\neq0$, $\eta(\Gamma(G))=0$; $a_{n}=0$, $\eta(\Gamma(G))\geq1$.
It is obvious that $a_{n-1}\not=0$ when $l=1$, thus   $\eta(\Gamma(G))=1$ when $l=1$.

{\bf Case 2: }  $p$ is even.

      Let $v$ be the vertex of $\Gamma(G)$ joining $\Gamma(C_{p})$ and $\Gamma(P_{l})$, then $v$ is a cut-point of $\Gamma(G)$. Note that $\eta(\Gamma(C_{p}-v))=\eta(\Gamma(P_{p-1}))=1$ by Proposition \ref{prop21} and $\eta(\Gamma(C_{p}))=2$ or $0$ by Theorem \ref{thm22}.

{\bf Subcase 2.1: } $\eta(\Gamma(C_{p}))=0$.

It's clear that $\eta(\Gamma(C_{p}-v))= \eta(\Gamma(C_{p}))+1$. Then by Theorem \ref{thm31} and Lemma \ref{lem11},

 $\eta(\Gamma(G))=\eta(\Gamma(G-v))-1$

 \hskip1.6cm $=\eta(\Gamma(C_{p}-v))+\eta(\Gamma(G-C_{p}))-1$

\hskip1.6cm $ =\eta(\Gamma(G-C_{p}))$

\hskip1.6cm $=\left\{\begin{array}{ll}
\eta(\Gamma(P_{q-1})),      & {\mbox{if } l \mbox{ is odd}; }\\
\eta(\Gamma(C_{q})),        & {\mbox{if } l \mbox{ is even}. }
\end{array} \right. $

\vskip.2cm

{\bf Subcase 2.2: } $\eta(\Gamma(C_{p}))=2$.

It's clear that $\eta(\Gamma(C_{p}-v))= \eta(\Gamma(C_{p}))-1$. Then by Theorem \ref{thm32} and Lemma \ref{lem11},

$\eta(\Gamma(G))=\eta(\Gamma(C_{p}-v))+\eta(\Gamma(G-C_{p}+v))$

\hskip1.6cm $=1+\eta(\Gamma(G-C_{p}+v))$

\hskip1.6cm $=\left\{
\begin{array}{llll}
1+\eta(\Gamma(C_{q})),         & {\mbox{if } l \mbox{ is odd}; }\\
1+\eta(\Gamma(P_{q-1})),       & {\mbox{if } l\mbox{ is even}. }
\end{array} \right. $

By Theorem \ref{thm22}, $ \eta(\Gamma(C_{q}))=\left\{
\begin{array}{lll}
0,                        & {\mbox{if } q \mbox{ is odd;}}\\
0 \mbox{ or }2,           & {\mbox{if }q \mbox{ is even.}}\\
\end{array} \right. $ Then by Proposition \ref{prop21} and above arguments,  (2) and (3) hold.
\end{proof}

%\begin{theo}\label{thm42}
%Let $p,q,l$ be integers with $p,q\geq3,l=1$ and $G=\infty(p,q,1))$.
%
%{\rm (1) } If  $p$ and $q$ are odd, then
%
%$$ \eta(\Gamma(G))=\left\{
%\begin{array}{lll}
%0,         & {\mbox{if } s(\Gamma(C_{p}))-s(\Gamma(C_{q}))+\frac{q-p}{2}\equiv \ 0(\ \bmod \ 2); }\\
%1,         & {\mbox{if } s(\Gamma(C_{p}))-s(\Gamma(C_{q}))+\frac{q-p}{2}\equiv \ 1(\ \bmod \ 2). }\\
%\end{array} \right. $$
%
%{\rm (2) } If $p$ and $q$ have different parities, without loss of generality, let $p$ be even, then
%
%$$ \eta(\Gamma(G))=\left\{
%\begin{array}{lll}
%0,      & {\mbox{if } \eta(\Gamma(C_{p}))=0; }\\
%1,     & {\mbox{if } \eta(\Gamma(C_{p}))=2. }\\
%\end{array} \right. $$
%
%{\rm (3) } If $p$ and $q$ are even,

%$$ \eta(\Gamma(G))=\left\{
%\begin{array}{lll}
%3,     & {\mbox{if } \eta(\Gamma(C_{p}))=\eta(\Gamma(C_{q}))=2; }\\
%1,     & {\mbox{if } \eta(\Gamma(C_{p}))\cdot\eta(\Gamma(C_{q}))=0.}\\
%\end{array} \right. $$
%\end{theo}

\section{The nullity set of bicyclic signed graphs}
\hskip.6cm Denoted by $\mathcal{B}_n$ the set of all bicyclic graphs on $n$ vertices.
Obviously, $\mathcal{B}_n$ consists of three types of graphs:
first type denoted by $B^+_n$ is the set of those graphs each of which is an $\infty$-graph, $\infty(p,q,l)$,  with trees attached when $l>1$;
second type denoted by $B^{++}_n$ is the set of those graphs each of which is an $\infty$-graph, $\infty(p,q,l)$,  with trees attached when $l=1$;
third type denoted by $\Theta_n$ is the set of those graphs each of which is a $\Theta$-graph, $\Theta(p,q,l)$,  with trees attached.
Then $\mathcal{B}_n=B_n^+\cup B_n^{++}\cup \Theta_n.$

 Let $\Gamma(\mathcal{B}_n)$ be the set of all bicyclic signed graphs on $n$ vertices.
 Clearly, $\Gamma(\mathcal{B}_n)=\Gamma(B_n^+)\cup \Gamma(B_n^{++})\cup \Gamma(\Theta_n).$

%Firstly we give some definitions used below: Let $\Gamma({\ss}{_{n}})$ be the set of all bicyclic signed graph on $n$ vertices. It's clear that $\Gamma({\ss}{_{n}})$ consists of three types of graphs. First type denoted by $\Gamma(B_{n}^{+})$ is the set of those graphs each of which is an signed $\infty$-graph with signed trees attached when $p>1$; Second type denoted by $\Gamma(B_{n}^{++})$ is the set of those graphs each of which is an signed $\infty$-graph with signed trees attached when $p=1$; Third type denoted by $\Gamma(\theta_{n})$ is the set of those graph each of which is an signed $\Theta$-graph with signed trees attached.

Let $\Gamma(G)=(G,\sigma)$ be a signed graph on $n$ vertices. Suppose $\theta: V(G)\longrightarrow\{+,-\}$ is a sign function.
Switching $\Gamma(G)$ by $\theta$ means forming a new signed graph $\Gamma(G)^{\theta}=(G, \sigma^{\theta})$
 whose underlying graph is the same as $G$, but whose sign function is defined on
an edge $uv$ by $\sigma^{\theta}(uv)=\theta(u)\sigma(uv)\theta(v)$.
Note that switching does not change the signs or balanceness of the cycles of $\Gamma(G)$.
If we define a diagonal signature matrix $D^{\theta}=diag(d_1, d_2, \ldots, d_n)$ with $d_i=\theta(v_i)$ for each $v_i\in V(G)$, then
$A(\Gamma(G)^{\theta})=D^\theta A(\Gamma(G))D^{\theta}.$
Two graphs $\Gamma_1(G), \Gamma_2(G)$ are called switching equivalent, denoted by $\Gamma_1(G)\sim\Gamma_2(G)$, if there
exists a switching function $\theta$ such that $\Gamma_2(G)=\Gamma_1^{\theta}(G)$,  or equivalently, $A(\Gamma_2(G))=D^\theta A(\Gamma_1(G))D^{\theta}.$

\begin{theo}\label{thm11}{\rm (\cite{2003})}
Let $\Gamma(G)=(G,\sigma)$ be a signed graph. Then $\Gamma(G)$ is balanced if and only if $\Gamma(G)=(G, \sigma)\sim (G,+).$
\end{theo}

Note that switching equivalence is a relation of equivalence, and two switching equivalent graphs
have the same nullity. Therefore, when we discuss the nullity of signed graphs, we can choose an arbitrary
representative of each switching equivalent class. %For any graphs, there are exactly two switching equivalent classes.
If a  signed graph is balanced, by Theorem \ref{thm11}, it is switching
equivalent to one with all edges positive, that is, the underlying graph. Thus we only need to consider the case of unbalanced.

Liu et al.\cite{2008a}, Chang et al.\cite{2008c} characterize the maximal nullity of bicyclic graphs and  determine the the nullity set of $\mathcal{B}_n$.
Recently, Fan et al.\cite{20132} characterize the maximal, the second maximal nullity of bicyclic signed graphs.

\begin{theo}\label{thm110}{\rm (\cite{2008c})} Let $n$ be a positive integer, $[0,n]=\{0,1,2,\ldots, n\}$. Then

{\rm (1) }  Let $n\geq 7$, the nullity set  of $B_n^+$ is $[0,n-6].$

{\rm (2) } Let $n\geq 8$, the nullity set  of $B_n^{++}$ is $[0,n-6].$

{\rm (3) } Let $n\geq 6$, the nullity set  of $\Theta_n$ is $[0,n-4].$
\end{theo}

In Subsection 5.1-Subsection 5.3,  we firstly obtain an upper bound of the nullity of bicyclic signed graphs in $\Gamma(B_n^+)$ and $\Gamma(B_n^{++})$,
then we obtain the nullity set of unbalanced bicyclic signed graphs in $\Gamma(B_n^+)$, $\Gamma(B_n^{++})$, $\Gamma(\Theta_n)$, respectively,
 and determine the nullity set of (unbalanced) bicyclic signed graphs.

\subsection{The nullity set of unbalanced  bicyclic signed graphs in $\Gamma(B_{n}^{+})$}
%\hskip.6cm In this subsection, we give an upper bound of the nullity of unbalanced bicyclic signed graphs in $\Gamma(B_{n}^{+})$,
%then we obtain the nullity set of unbalanced bicyclic signed graphs in  $\Gamma(B_{n}^{+})$.

\begin{theo}\label{thm51}
Let $n\geq7$, $\Gamma(G)\in\Gamma(B_{n}^{+})$. Then $\eta(\Gamma(G))\leq n-6 $.
\end{theo}
\begin{proof}
Let $G\in B_{n}^{+}$ be a bicyclic graph with trees attached on a $\infty$-graph, $\infty(p,q,l)$, where $p,q\geq 3, l\geq 2$.

{\bf Case 1: } $p,q\in\{3,4\}$.

{\bf Subcase 1.1: } $p=q=4$.

Note that $p+q+l-2=6+l\geq \left\{
\begin{array}{lll}
9,           & {\mbox{if } l \mbox{ is odd;}}\\
8,           & {\mbox{if }l \mbox{ is even.}}\\
\end{array} \right.$ Then  by (3) of Theorem \ref{thm41}, $r(\Gamma(\infty(4,4,l)))\geq 6$.

Clearly, $\infty(4,4,l)$ is an induced subgraph of $G$, then  $r(\Gamma(G))\geq r(\Gamma(\infty(4,4,l)))\geq 6$. Therefore $\eta(\Gamma(G))\leq n-6$.

{\bf Subcase 1.2: } $p\not=4$ or $q\not=4$.

In this case, there must exist a graph $H$ on $6$ vertices as an induced subgraph of $G$, where  $H=H_{1}$ or $H=H_{2}$ shown in Fig.1.
By Lemma \ref{lem11} repeatedly we obtain $\eta(\Gamma(H_{1}))=\eta(\Gamma(H_{2}))=0$, then $r(\Gamma(H_{1}))=r(\Gamma(H_{2}))=6$.
 Thus  $r(\Gamma(G))\geq r(\Gamma(H))\geq 6$ and $\eta(\Gamma(G))\leq n-6$.

%\vskip.3cm\vspace{0.6cm}

\begin{picture}(300,100)

\put(20,40){\circle*{2}}
\put(40,80){\circle*{2}}
\put(60,40){\circle*{2}}
\put(80,40){\circle*{2}}
\put(100,40){\circle*{2}}
\put(120,40){\circle*{2}}

\put(20,40){\line(1,0){40}}
\put(20,40){\line(1,2){20}}
\put(40,80){\line(1,-2){20}}
\put(60,40){\line(1,0){40}}
\put(80,40){\line(1,0){40}}
\put(100,40){\line(1,0){20}}

\put(160,40){\circle*{2}}
\put(180,80){\circle*{2}}
\put(200,40){\circle*{2}}
\put(200,80){\circle*{2}}
\put(220,40){\circle*{2}}
\put(240,40){\circle*{2}}

\put(160,40){\line(1,0){40}}
\put(160,40){\line(1,2){20}}
\put(180,80){\line(1,-2){20}}
\put(200,40){\line(1,0){40}}
\put(200,40){\line(0,1){40}}
\put(220,40){\line(1,0){20}}

\put(300,60){\circle{50}}
\put(320,60){\circle*{2}}
\put(360,60){\circle*{2}}

\put(320,60){\line(1,0){40}}

\put(60,20){$H_{1}$}
\put(200,20){$H_{2}$}

\put(290,50){$C_{p}$}
\put(300,20){$H_{3}$}
\put(180,-5){Fig.1}
\end{picture}

{\bf Case 2: } $p\geq5$ or $q\geq5$.

Without loss of generality, we assume that $p\geq5$. There must exist a graph $H_{3}$ on $p+1$ vertices shown in Fig.1 as an induced subgraph of $G$.
By Lemma \ref{lem11} and Propsition \ref{prop21}, it is easy to check that
$ \eta(\Gamma(H_{3}))=\left\{
\begin{array}{lll}
0,      & {\mbox{if } p \mbox{ is odd;}}\\
1,        & {\mbox{if } p \mbox{ is even.}}\\
\end{array} \right. $
Hence
$$ r(\Gamma(H_3))=\left\{
\begin{array}{lll}
p+1,        & {\mbox{if } p \mbox{ is odd;}}\\
p,         & {\mbox{if } p \mbox{ is even.}}\\
\end{array} \right. $$
Since $p\geq5$, $r(\Gamma(H_{3}))\geq6$. Then $r(\Gamma(G))\geq r(\Gamma(H_{3}))\geq6$. Thus $\eta(\Gamma(G))\leq n-6$.
\end{proof}

\begin{theo}\label{thm52}
Let $n\geq7$. Then the nullity set of unbalanced bicyclic signed graphs in $\Gamma(B_{n}^{+})$ is $[0,n-6]$.
\end{theo}
\begin{proof}
It suffices to show that for each $k\in [0,n-6]$, there exists an unbalanced bicyclic signed graph $\Gamma(G)\in \Gamma(B_{n}^{+})$ such that $\eta(\Gamma(G))=k$.

{\bf Case 1: } $k=0$.

It's clear that there exists an unbalanced bicyclic signed graph $\Gamma(G)=\Gamma(\infty(p,q,l))\in \Gamma(B_{n}^{+})$ satisfying $\eta(\Gamma(G))=0$ by Theorem \ref{thm41}, where $p,q\geq 3,l\geq 2$.

{\bf Case 2: } $k=n-6$.

Let $G=G_{1}$ shown in Fig.2, where $\Gamma(G_{1})$ contains a balanced quadrangle and an unbalanced triangle.
Thus $\eta(\Gamma(C_{4}))=2$ by Theorem \ref{thm22}.

If $n=7$, then $G=\infty(3,4,2)$ and $\eta(\Gamma(G))=1=n-6$ by (2) of Theorem \ref{thm41}.

If $n\geq 8$, then by Lemma \ref{lem11} and Lemma \ref{lem31}, we have
$\eta(\Gamma(G))=\eta(\Gamma(P_{2})\cup \Gamma(C_{4})\cup(n-8)\Gamma(K_{1}))=\eta(\Gamma(P_{2}))+\eta(\Gamma(C_{4}))+(n-8)\eta(\Gamma(K_{1}))=0+2+(n-8)=n-6$.

\vspace{0.6cm}

\begin{picture}(380, 100)

\put(20,60){\circle*{2}}
\put(50,40){\circle*{2}}
\put(40,100){\circle*{2}}
\put(60,60){\circle*{2}}
\put(70,40){\circle*{2}}
\put(100,60){\circle*{2}}
\put(100,100){\circle*{2}}
\put(140,60){\circle*{2}}
\put(140,100){\circle*{2}}

\put(20,60){\line(1,0){40}}
\put(20,60){\line(1,2){20}}
\put(40,100){\line(1,-2){20}}
\put(60,60){\line(-1,-2){10}}
\put(60,60){\line(1,-2){10}}
\put(60,60){\line(1,0){40}}
\put(100,60){\line(1,0){40}}
\put(100,60){\line(0,1){40}}
\put(100,100){\line(1,0){40}}
\put(140,100){\line(0,-1){40}}

\put(70,0){$G_{1}$}
\put(53,40){$\ldots$}
\put(60,50){\circle{150}}
\put(50,20){$S_{n-6}$}

\put(180,60){\circle*{2}}
\put(210,40){\circle*{2}}
\put(200,100){\circle*{2}}
\put(260,60){\circle*{2}}
\put(230,40){\circle*{2}}
\put(300,60){\circle*{2}}
\put(340,60){\circle*{2}}
\put(360,100){\circle*{2}}
\put(380,60){\circle*{2}}

\put(180,60){\line(1,0){40}}
\put(180,60){\line(1,2){20}}
\put(200,100){\line(1,-2){20}}
\put(210,40){\line(1,2){10}}
\put(220,60){\line(1,-2){10}}
\put(220,60){\line(1,0){40}}
\put(300,60){\line(1,0){40}}
\put(340,60){\line(1,2){20}}
\put(340,60){\line(1,0){40}}
\put(360,100){\line(1,-2){20}}

\put(290,0){$G_{2}$}
\put(225,20){$S_{k+2}$}
\put(213,40){$\ldots$}
\put(272,60){$\ldots$}
\put(255,55){$\underbrace{ \quad \quad\quad\quad }$}
\put(260,35){$P_{n-k-7}$}
\put(220,50){\circle{150}}
\put(160,-10){Fig.2}
\end{picture}
\vskip.3cm
{\bf Case 3: } $1\leq k\leq n-7$.

Let $G=G_{2}$ shown in Fig.2, where $\Gamma(G_{2})$ contains two unbalanced triangles. By using Lemma \ref{lem11} repeatedly, after $\lfloor\frac{n-k-4}{2}\rfloor$ steps, we have $$\eta(\Gamma(G))=\left\{
\begin{array}{lll}
\eta(\Gamma(P_{2})\cup \Gamma(C_{3}) \cup k\Gamma(K_{1})),        & {\mbox{if } n-k \mbox{ is odd;}}\\
\eta(2\Gamma(P_{2}) \cup k\Gamma(K_{1})),         & {\mbox{if } n-k \mbox{ is even.}}\\
\end{array} \right. $$

%\eta(\Gamma(P_{2})\cup \Gamma(C_{3}) \cup k\Gamma(K_{1}))$.
Hence by  Lemma \ref{lem31},

$\eta(\Gamma(G))=\left\{
\begin{array}{lll}
\eta(\Gamma(P_{2}))+\eta(\Gamma(C_{3}))+k\eta(\Gamma(K_{1}))=0+0+k=k,        & {\mbox{if } n-k \mbox{ is odd;}}\\
2\eta(\Gamma(P_{2}))+k\eta(\Gamma(K_{1}))=0+k=k,         & {\mbox{if } n-k \mbox{ is even.}}\\
\end{array} \right. $
\end{proof}

\subsection{The nullity set of unbalanced  bicyclic signed graphs in $\Gamma(B_{n}^{++})$}
%In this section, we introduce the upper bound of $\Gamma(B_{n}^{++})$, then we get the nullity set of unbalanced $\Gamma(B_{n}^{++})$.

\begin{theo}\label{thm53}
Let $n\geq 8$, $\Gamma(G)\in\Gamma(B_{n}^{++})$. Then $\eta(\Gamma(G))\leq n-6 $.
\end{theo}
\begin{proof}
Let $G\in B_{n}^{++}$ be a bicyclic graph with trees attached on a $\infty$-graph, $\infty(p,q,1)$, where $p,q\geq 3$.

{\bf Case 1: } $p,q\in\{3,4\}$.

In this case, there must exist a graph $H$ on $h$ vertices as an induced subgraph of $G$, where  $H=H_{4}, H_5$ with $h=6$,
or $H=H_{6}, H_7, H_8, H_9$ with $h=7$, or  $H=H_{10}, H_{11}, H_{12}$ with $h=8$ shown in Fig.3.
By Lemma \ref{lem11} repeatedly and Theorem \ref{thm22}, we obtain  $\eta(\Gamma(H_{4}))=\eta(\Gamma(H_{5}))=\eta(\Gamma(H_{8}))=0$, $\eta(\Gamma(H_{6}))=\eta(\Gamma(H_{7}))=\eta(\Gamma(H_{9}))=1$, $\eta(\Gamma(H_{11}))=\eta(\Gamma(H_{12}))=2$, and let $\Gamma(C_{4})$ is the quadrangle containing no pendent edge of $\Gamma(H_{10})$, we have
$$ \eta(\Gamma(H_{10}))=\left\{
\begin{array}{lll}
2,        & {\mbox{if } \Gamma(C_{4}) \mbox{ is balanced;}}\\
0,        & {\mbox{if } \Gamma(C_{4}) \mbox{ is unbalanced.}}\\
\end{array} \right. $$

Hence for each $\Gamma(H_{i})(i=4, 5, \ldots, 12)$, we have $r(\Gamma(H_{i})
)\geq6$, so $r(\Gamma(G))\geq r(\Gamma(H_{i}))\geq 6$. Thus $\eta(\Gamma(G))\leq n-6$.

\vskip.5cm

\begin{picture}(360, 80)
\put(20,40){\circle*{2}}
\put(40,80){\circle*{2}}
\put(60,40){\circle*{2}}
\put(80,80){\circle*{2}}
\put(100,40){\circle*{2}}
\put(100,80){\circle*{2}}

\put(20,40){\line(1,2){20}}
\put(20,40){\line(1,0){40}}
\put(40,80){\line(1,-2){20}}
\put(60,40){\line(1,2){20}}
\put(60,40){\line(1,0){40}}
\put(80,80){\line(1,-2){20}}
\put(100,40){\line(0,1){40}}

\put(160,40){\circle*{2}}
\put(180,80){\circle*{2}}
\put(200,40){\circle*{2}}
\put(200,80){\circle*{2}}
\put(220,80){\circle*{2}}
\put(240,40){\circle*{2}}

\put(160,40){\line(1,2){20}}
\put(160,40){\line(1,0){40}}
\put(180,80){\line(1,-2){20}}
\put(200,40){\line(1,2){20}}
\put(200,40){\line(0,1){40}}
\put(200,40){\line(1,0){40}}
\put(220,80){\line(1,-2){20}}

\put(300,80){\circle*{2}}
\put(300,40){\circle*{2}}
\put(320,80){\circle*{2}}
\put(340,40){\circle*{2}}
\put(340,80){\circle*{2}}
\put(380,40){\circle*{2}}
\put(380,80){\circle*{2}}

\put(300,80){\line(0,-1){40}}
\put(300,40){\line(1,0){40}}
\put(300,40){\line(1,2){20}}
\put(320,80){\line(1,-2){20}}
\put(340,40){\line(0,1){40}}
\put(340,40){\line(1,0){40}}
\put(340,80){\line(1,0){40}}
\put(380,80){\line(0,-1){40}}

\put(50,20){$H_4$}
\put(190,20){$H_5$}
\put(330,20){$H_6$}
\end{picture}\\

\begin{picture}(400, 80)
\put(20,40){\circle*{2}}
\put(40,80){\circle*{2}}
\put(60,40){\circle*{2}}
\put(60,80){\circle*{2}}
\put(50,80){\circle*{2}}
\put(100,40){\circle*{2}}
\put(100,80){\circle*{2}}

\put(20,40){\line(1,2){20}}
\put(20,40){\line(1,0){40}}
\put(40,80){\line(1,-2){20}}
\put(60,40){\line(0,1){40}}
\put(60,40){\line(-1,4){10}}
\put(60,40){\line(1,0){40}}
\put(100,80){\line(0,-1){40}}
\put(60,80){\line(1,0){40}}

\put(160,40){\circle*{2}}
\put(180,80){\circle*{2}}
\put(200,40){\circle*{2}}
\put(200,80){\circle*{2}}
\put(240,80){\circle*{2}}
\put(240,40){\circle*{2}}
\put(260,80){\circle*{2}}

\put(160,40){\line(1,2){20}}
\put(160,40){\line(1,0){40}}
\put(180,80){\line(1,-2){20}}
\put(200,40){\line(0,1){40}}
\put(200,40){\line(1,0){40}}
\put(240,80){\line(0,-1){40}}
\put(240,80){\line(-1,0){40}}
\put(240,40){\line(1,2){20}}

\put(300,40){\circle*{2}}
\put(320,80){\circle*{2}}
\put(340,40){\circle*{2}}
\put(340,80){\circle*{2}}
\put(380,40){\circle*{2}}
\put(400,40){\circle*{2}}

\put(300,40){\line(1,0){40}}
\put(300,40){\line(1,2){20}}
\put(320,80){\line(1,-2){20}}
\put(340,40){\line(0,1){40}}
\put(340,40){\line(1,0){40}}
\put(340,80){\line(1,0){40}}
\put(380,80){\line(0,-1){40}}
\put(380,80){\line(1,-2){20}}

\put(50,20){$H_7$}
\put(190,20){$H_8$}
\put(330,20){$H_9$}
\end{picture}\\

\begin{picture}(360, 120)
\put(20,80){\circle*{2}}
\put(20,120){\circle*{2}}
\put(100,120){\circle*{2}}
\put(60,40){\circle*{2}}
\put(60,80){\circle*{2}}
\put(60,120){\circle*{2}}
\put(100,40){\circle*{2}}
\put(100,80){\circle*{2}}

\put(20,80){\line(0,1){40}}
\put(20,80){\line(1,0){40}}
\put(100,120){\line(0,-1){40}}
\put(20,120){\line(1,0){40}}
\put(60,120){\line(0,-1){40}}
\put(60,80){\line(1,0){40}}
\put(60,80){\line(0,-1){40}}
\put(60,40){\line(1,0){40}}
\put(100,80){\line(0,-1){40}}

\put(160,80){\circle*{2}}
\put(180,80){\circle*{2}}
\put(180,120){\circle*{2}}
\put(220,40){\circle*{2}}
\put(220,80){\circle*{2}}
\put(220,120){\circle*{2}}
\put(260,40){\circle*{2}}
\put(260,80){\circle*{2}}

\put(160,80){\line(1,2){20}}
\put(180,80){\line(1,0){40}}
\put(180,80){\line(0,1){40}}
\put(180,120){\line(1,0){40}}
\put(220,120){\line(0,-1){40}}
\put(220,80){\line(1,0){40}}
\put(220,80){\line(0,-1){40}}
\put(220,40){\line(1,0){40}}
\put(260,40){\line(0,1){40}}

\put(300,80){\circle*{2}}
\put(300,120){\circle*{2}}
\put(340,40){\circle*{2}}
\put(340,80){\circle*{2}}
\put(340,120){\circle*{2}}
\put(380,40){\circle*{2}}
\put(380,80){\circle*{2}}
\put(380,120){\circle*{2}}

\put(300,80){\line(1,0){40}}
\put(300,80){\line(0,1){40}}
\put(300,120){\line(1,0){40}}
\put(340,80){\line(1,0){40}}
\put(340,80){\line(0,1){40}}
\put(340,80){\line(1,1){40}}
\put(340,80){\line(0,-1){40}}
\put(340,40){\line(1,0){40}}
\put(380,40){\line(0,1){40}}

\put(50,20){$H_{10}$}
\put(210,20){$H_{11}$}
\put(330,20){$H_{12}$}
\put(200,-7) {Fig.3}
\end{picture}

{\bf Case 2: } $p\geq5$ or $q\geq5$.

Without loss of generality, we assume that $p\geq5$. There must exist a graph $H_{3}$ on $p+1$ vertices shown in Fig.1 as an induced subgraph of $G$.
 Similar to the proof of Case 2 in Theorem \ref{thm51}, we have $r(\Gamma(H_{3}))\geq6$.
Then $r(\Gamma(G))\geq r(\Gamma(H_{3}))\geq6$ and  thus $\eta(\Gamma(G))\leq n-6$.
\end{proof}

\begin{lem}\label{lem51}{ \rm (\cite{20132})}
Let $H_{13}$ be a  graph on $5$ vertices as shown in Fig.4, and the two triangles of $\Gamma(H_{13})$ have the same balanceness. Then $\eta(\Gamma(H_{13}))=n-5=0$.
\end{lem}

\vskip.3cm

\begin{picture}(300, 80)
\put(200,40){\circle*{2}}
\put(220,80){\circle*{2}}
\put(240,40){\circle*{2}}
\put(260,80){\circle*{2}}
\put(280,40){\circle*{2}}
\put(280,40){\circle*{2}}
\put(300,50){\circle*{2}}
\put(300,30){\circle*{2}}

\put(200,40){\line(1,2){20}}
\put(200,40){\line(1,0){40}}
\put(220,80){\line(1,-2){20}}
\put(240,40){\line(1,2){20}}
\put(240,40){\line(1,0){40}}
\put(260,80){\line(1,-2){20}}
\put(280,40){\line(2,1){20}}
\put(280,40){\line(2,-1){20}}

\put(240,15){$G_{3}$}
\put(298,35){$\vdots$}
\put(290,40){\circle{80}}
\put(280,65){$S_{n-4}$}

\put(60,40){\circle*{2}}
\put(80,80){\circle*{2}}
\put(100,40){\circle*{2}}
\put(120,80){\circle*{2}}
\put(140,40){\circle*{2}}

\put(60,40){\line(1,2){20}}
\put(60,40){\line(1,0){40}}
\put(80,80){\line(1,-2){20}}
\put(100,40){\line(1,2){20}}
\put(100,40){\line(1,0){40}}
\put(120,80){\line(1,-2){20}}
\put(90,15){$H_{13}$}
\put(150, 5){Fig.4}
\end{picture}

\begin{theo}\label{thm54}
Let $n\geq 8$, Then the nullity set of unbalanced bicyclic signed graphs in $\Gamma(B_{n}^{++})$ is $[0,n-6]$.
\end{theo}
\begin{proof}
It suffices to show that for each $k\in [0,n-6]$, there exists an unbalanced bicyclic signed graph $\Gamma(G)\in \Gamma(B_{n}^{++})$ such that $\eta(\Gamma(G))=k$.

{\bf Case 1: } $k=0$.

It's clear that there exists an unbalanced bicyclic signed graph $\Gamma(G)=\Gamma(\infty(p,q,1))\in \Gamma(B_{n}^{++})$  satisfying $\eta(\Gamma(G))=0$ by Theorem \ref{thm41}, where $p,q\geq 3$.

{\bf Case 2: } $k=n-6$.

Let $G=G_{3}$ shown in Fig.4, where the two triangles of $\Gamma(G_{3})$ are unbalanced.
Then by Lemma \ref{lem11} and Lemma \ref{lem31}, we have
$\eta(\Gamma(G))=\eta(\Gamma(P_{2})\cup (n-6)\Gamma(K_1))=\eta(\Gamma(P_{2}))+(n-6)\eta(\Gamma(K_{1}))=0+(n-6)=n-6$.

\vskip.5cm
\begin{picture}(360, 80)
\put(60,40){\circle*{2}}
\put(80,80){\circle*{2}}
\put(100,40){\circle*{2}}
\put(120,80){\circle*{2}}
\put(140,40){\circle*{2}}
\put(180,40){\circle*{2}}
\put(220,40){\circle*{2}}
\put(260,40){\circle*{2}}
\put(280,50){\circle*{2}}
\put(280,30){\circle*{2}}

\put(60,40){\line(1,2){20}}
\put(60,40){\line(1,0){40}}
\put(80,80){\line(1,-2){20}}
\put(100,40){\line(1,2){20}}
\put(100,40){\line(1,0){40}}
\put(120,80){\line(1,-2){20}}
\put(140,40){\line(1,0){40}}
\put(220,40){\line(1,0){40}}
\put(260,40){\line(2,1){20}}
\put(260,40){\line(2,-1){20}}

\put(130,10){$G_{4}$}
\put(278,35){$\vdots$}
\put(270,40){\circle{80}}
\put(260,65){$S_{k+2}$}
\put(195,40){$\ldots$}
\put(177,35){$\underbrace{\quad \quad \quad\quad}$}
\put(180,10){$P_{n-k-7}$}

\put(180,-10){Fig.5}
\end{picture}

{\bf Case 3: } $1\leq k\leq n-7$.

Let $G=G_{4}$ shown in Fig.5, where the two triangles of $\Gamma(G_{4})$ are unbalanced.

{\bf Subcase 3.1: } $n-k$ is odd.

By using Lemma \ref{lem11} repeatedly, after $\frac{n-k-5}{2}$ steps, we obtain the graph $k\Gamma(K_{1})\cup\Gamma(H_{13})$, where  the two triangles of $\Gamma(H_{13})$ are unbalanced.   Hence $\eta(\Gamma(G))=\eta(k\Gamma(K_{1})\cup \Gamma(H_{13}))=\eta(k\Gamma(K_{1}))+\eta(\Gamma(H_{13}))=k+0=k$ by Lemma \ref{lem51}.

{\bf Subcase 3.2: } $n-k$ is even.

By using Lemma \ref{lem11} repeatedly, after $\frac{n-k-2}{2}$ steps, we obtain the graph $k\Gamma(K_{1})\cup \Gamma(P_{2})$. Hence $\eta(\Gamma(G))=\eta(k\Gamma(K_{1})\cup \Gamma(P_{2}))=\eta(k\Gamma(K_{1}))+\eta(\Gamma(P_{2}))=k+0=k$.
\end{proof}

\subsection{The nullity set of unbalanced  bicyclic signed graphs in $\Gamma(\Theta_{n})$}

\begin{lem}\label{lem52} { \rm (\cite{20132})}
Let $\Gamma(G)$ be an unbalanced bicyclic signed graph on $n$ vertices. Then $\eta(\Gamma(G))\leq n-3$, with equality if and only if $\Gamma(G)=\Gamma(\Theta(2,2,1))$ and the two triangles of~  $\Gamma(\Theta(2,2,1))$ are both unbalanced.
\end{lem}

By Lemma \ref{lem52}, we obtain the following result immidiately.

\begin{prop}\label{prop51}
Let $n\geq 5$, $\Gamma(G)$ be an unbalanced bicyclic signed graph  in $\Gamma(\Theta_{n})$. Then $\eta(\Gamma(G))\leq n-4$.
\end{prop}

\begin{theo}\label{thm56}
Let $n\geq 6$. Then the nullity set of unbalanced bicyclic signed graphs  in $\Gamma(\Theta_{n})$ is $[0,n-4]$.
\end{theo}
\begin{proof}
 It suffices to show that for each $k\in [0,n-4]$, there exists an unbalanced bicyclic signed graph $\Gamma(G)\in \Gamma(\Theta_n)$ such that $\eta(\Gamma(G))=k$.

{\bf Case 1: } $k=0$.

Let $G=G_{5}$ shown in Fig.6, where $\Gamma(G_{5})$ contains at least an unbalanced triangle. By Lemma \ref{lem11} and Theorem \ref{thm22} (when $n$ is odd) or Proposition \ref{prop21}  (when $n$ is even), we have $\eta(\Gamma(G))=0$.

{\bf Case 2: } $k=n-4$.

Let $G=G_{6}$ shown in Fig.6, where $\Gamma(G_{6})$ contains at least an unbalanced triangle. By Lemma \ref{lem11} and Proposition \ref{prop21}, we have $\eta(\Gamma(G))=\eta(\Gamma(P_3)\cup (n-5)\Gamma(K_1))=\eta(\Gamma(P_3))+(n-5)\eta(\Gamma(K_1))=1+(n-5)=n-4$.

{\bf Case 3: } $1\leq k\leq n-5$.

{\bf Subcase 3.1: } $n-k$ is odd.

Let $G=G_{7}$ shown in Fig.6, where the two triangles of $G_{7}$ are unbalanced. By using Lemma \ref{lem11} repeatedly, after $\frac{n-k-3}{2}$ steps, we obtain the graph $\Gamma(\Theta(2,2,1))\cup (k-1)\Gamma(K_{1})$. Hence $\eta(\Gamma(G))=\eta(\Gamma(\Theta(2,2,1))\cup (k-1)\Gamma(K_{1}))=\eta(\Gamma(\Theta(2,2,1)))+(k-1)\eta(\Gamma(K_{1}))=1+(k-1)=k$ by Lemma \ref{lem11} and Lemma \ref{lem52}.

\vskip.3cm

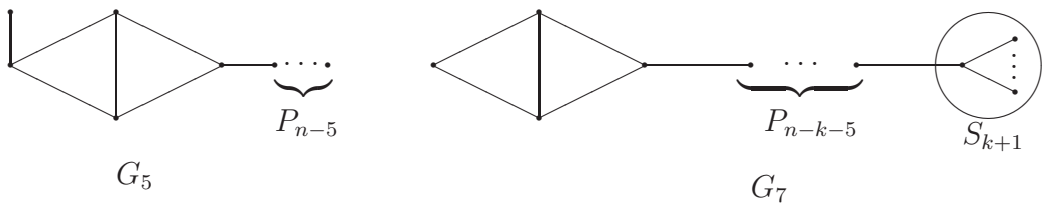
\begin{picture}(400, 80)
\put(20,60){\circle*{2}}
\put(20,80){\circle*{2}}
\put(60,40){\circle*{2}}
\put(60,80){\circle*{2}}
\put(100,60){\circle*{2}}
\put(140,60){\circle*{2}}
\put(120,60){\circle*{2}}

\put(20,60){\line(2,1){40}}
\put(20,60){\line(0,1){20}}
\put(20,60){\line(2,-1){40}}
\put(60,40){\line(2,1){40}}
\put(60,40){\line(0,1){40}}
\put(60,80){\line(2,-1){40}}
\put(100,60){\line(1,0){20}}

\put(180,60){\circle*{2}}
\put(220,80){\circle*{2}}
\put(220,40){\circle*{2}}
\put(260,60){\circle*{2}}
\put(300,60){\circle*{2}}
\put(340,60){\circle*{2}}
\put(380,60){\circle*{2}}
\put(400,50){\circle*{2}}
\put(400,70){\circle*{2}}

\put(180,60){\line(2,1){40}}
\put(180,60){\line(2,-1){40}}
\put(220,80){\line(2,-1){40}}
\put(220,80){\line(0,-2){40}}
\put(220,40){\line(2,1){40}}
\put(260,60){\line(1,0){40}}
\put(340,60){\line(1,0){40}}
\put(380,60){\line(2,-1){20}}
\put(380,60){\line(2,1){20}}

\put(60,15){$G_{5}$}
\put(123,60){$\ldots$}
\put(120,55){$\underbrace{\quad }$}
\put(120,35){$P_{n-5}$}

\put(300,10){$G_{7}$}
\put(398,55){$\vdots$}
\put(390,60){\circle{60}}
\put(380,30){$S_{k+1}$}
\put(312,60){$\ldots$}
\put(295,55){$\underbrace{\quad \quad \quad \quad }$}
\put(305,35){$P_{n-k-5}$}
\end{picture}\\

\begin{picture}(400, 80)
\put(20,60){\circle*{2}}
\put(60,40){\circle*{2}}
\put(60,80){\circle*{2}}
\put(100,60){\circle*{2}}

\put(20,60){\line(2,1){40}}
\put(20,60){\line(2,-1){40}}
\put(60,80){\line(2,-1){40}}
\put(60,80){\line(0,-2){40}}
\put(60,40){\line(2,1){40}}
\put(60,80){\line(2,1){30}}
\put(60,80){\line(-2,1){30}}
\put(50,95){\circle*{2}}
\put(60,95){\circle*{2}}
\put(70,95){\circle*{2}}

\put(180,40){\circle*{2}}
\put(180,80){\circle*{2}}
\put(220,40){\circle*{2}}
\put(220,80){\circle*{2}}
\put(260,60){\circle*{2}}
\put(300,60){\circle*{2}}
\put(340,60){\circle*{2}}
\put(380,60){\circle*{2}}
\put(400,50){\circle*{2}}
\put(400,70){\circle*{2}}

\put(180,40){\line(0,1){40}}
\put(180,40){\line(1,0){40}}
\put(180,80){\line(1,0){40}}
\put(220,40){\line(2,1){40}}
\put(220,40){\line(0,1){40}}
\put(220,80){\line(2,-1){40}}
\put(260,60){\line(1,0){40}}
\put(340,60){\line(1,0){40}}
\put(380,60){\line(2,1){20}}
\put(380,60){\line(2,-1){20}}

\put(60,10){$G_{6}$}
\put(300,10){$G_{8}$}
\put(398,55){$\vdots$}
\put(390,60){\circle{60}}
\put(380,30){$S_{k}$}
\put(312,60){$\ldots$}
\put(295,55){$\underbrace{\quad \quad \quad \quad }$}
\put(305,35){$P_{n-k-5}$}

\put(200,-10){Fig.6}
\end{picture}

{\bf Subcase 3.2: } $n-k$ is even.

Let $G=G_{8}$ shown in Fig.6, where the  triangle of $G_{8}$ is unbalanced and the quadrangle is balanced.
 By using Lemma \ref{lem11} repeatedly, after $\frac{n-k-2}{2}$ steps, we obtain the graph $\Gamma(C_4)\cup (k-2)\Gamma(K_{1})$.
  Hence $\eta(\Gamma(G))=\eta(\Gamma(C_4)\cup (k-2)\Gamma(K_{1}))=\eta(\Gamma(C_4))+(k-2)\eta(\Gamma(K_{1}))=2+(k-2)=k$ by Lemma \ref{lem11} and Theorem \ref{thm22}.
\end{proof}

\subsection{In conclusion}

\hskip.6cm From the above discussion, by Theorem \ref{thm52}, Theorem \ref{thm54}, Theorem \ref{thm56} and Theorem \ref{thm110}, we can obtain the following results immediately.

 \begin{theo}\label{thm57}
Let $n\geq 8$. Then the nullity set of unbalanced bicyclic signed graphs  is $[0,n-4]$.
\end{theo}

 \begin{theo}\label{thm58}
Let $n\geq 8$. Then the nullity set of  bicyclic signed graphs  is $[0,n-4]$.
\end{theo}

When $4\leq n\leq 7$, the nullity set of  bicyclic signed graph is easy to obtain by known results and directly calculating, so we omit it.

\end{document}